\documentclass[12pt]{amsart}
\usepackage{graphicx}
\usepackage{amscd}
\usepackage[top=1in,bottom=1in,left=1.15in,right=1.15in]{geometry} 
\usepackage{amsmath,amssymb,amsfonts}
\usepackage{amsthm}
\usepackage{color}
\usepackage{hyperref}
\newtheorem{theo}{Theorem}[section]
\newtheorem{prop}[theo]{Proposition}
\newtheorem{lemma}[theo]{Lemma}
\newtheorem{defn}[theo]{Definition}
\newtheorem{rem}[theo]{Remark}
\newtheorem{cor}[theo]{Corollary}

\newtheorem{con}[theo]{Conjecture}

\begin{document}
\title{H(2)-unknotting operation related to 2-bridge links}
\author{Yuanyuan Bao}
\address{
Graduate school of Mathematics,
Tokyo Institute of Technology,
Oh-okayama, Meguro, Tokyo 152-8551, Japan
}
\email{bao.y.aa@m.titech.ac.jp}
\date{}
\begin{abstract}
This paper concerns the H(2)-unknotting numbers of links related to 2-bridge links. It consists of three parts. In the first part, we consider a necessary and sufficient condition for a 2-bridge link to have H(2)-unknotting number one. The second part concerns an explicit form of composite links with H(2)-unknotting number one. In the last part, we develop a method of studying the H(2)-unknotting numbers of some tangle unknotting number one knots via 2-bridge knots.

\end{abstract}
\thanks{The author is supported by scholarship from
the Ministry of Education, Culture, Sports, Science and Technology of Japan.}
\maketitle

\section{Introduction}
In this paper all the links are assumed to be unoriented except otherwise stated.
An H(2)-move is a local transformation between link diagrams, as shown in Figure~\ref{fig:f2}. It is an unknotting operation.
The {\it H(2)-unknotting number} \cite{MR1075165} of a link $L$ is the minimal number of H(2)-moves needed to change the link into the unknot, and we denote it by $u_{2}(L)$. In this paper, we study this unknotting operation for links related to 2-bridge links. It consists of three parts.
\begin{figure}[h]
	\centering
		\includegraphics[width=0.5\textwidth]{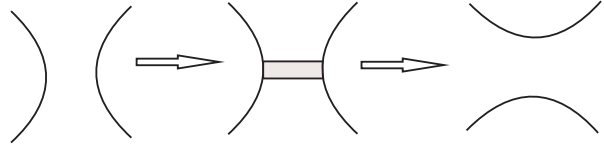}
	\caption{An H(2)-move.}
	\label{fig:f2}
\end{figure}

In Section~2, we make some observations about 2-bridge links with H(2)-unknotting number one. Our purpose is to find out 2-bridge links whose H(2)-unknotting numbers are one. Here is our main observation.
\begin{prop}
\label{th1}
The 2-bridge link $S(p, q)$ has H(2)-unknotting number one if and only if the lens space $L(p, q)$ can be obtained as an integral surgery along a Berge knot in $S^{3}$.
\end{prop}
We remark that this observation is just a corollary of some known results about integral surgery. Since our interest is H(2)-unknotting operation, we think it is worth to write it down. In \cite{MR2573402}, an incomplete table of H(2)-unknotting numbers of knots is provided. Among knots with nine crossings, there are six knots whose H(2)-unknotting numbers are unknown. We confirm that among them the 2-bridge knots $9_{21}, 9_{23}, 9_{26}$ and $9_{31}$ are knots with H(2)-unknotting number one. We refer to Rolfson's table for the notations of knots.

\begin{cor}
\label{table}
The 2-bridge knots $9_{21}, 9_{23}, 9_{26}$ and $9_{31}$ have H(2)-unknotting number one.
\end{cor}
Section~2 essentially has no direct relation to Sections 3 and 4.

Section~3 is about composite links of H(2)-unknotting number one. Bleiler \cite{MR1045396} and Eudave-Mu$\tilde{\rm n}$oz \cite{MR1112545} proved that if a composite link has H(2)-unknotting number one then it is a connected sum of a 2-bridge link and a prime link. The purpose of Section~3 is to study the explicit forms of the two summands of such composite links. First we have the following result:

\begin{prop}
\label{main}
If $K_{1}$ is the 2-bridge link $S(q, p)$, and $K_{2}$ is a $(p,q)$-tangle unknotting number one link, then the H(2)-unknotting number of the composite $K_{1}\sharp K_{2}$ is one.
\end{prop}
We conjecture that the converse holds as well, which we expressed as Conjecture~\ref{con}. When we restrict the two summands to 2-bridge links, we have the following complete description.
\begin{prop}
\label{main2}
Suppose $S(p,q)$ and $S(r,s)$ are two non-trivial 2-bridge links. Then the composite $S(p,q)\sharp S(r,s)$ has H(2)-unknotting number one if and only if either $S(r,s)=S(q,p)$, or $S(p,q)=S(v,\epsilon)$ and $S(r,s)=S(vab+\epsilon, va^{2})$ for $\epsilon = \pm 1$ and some integers $v$, $a$ and $b$ satisfying $(a,b)=1$.
\end{prop}

Notice that the 2-bridge link $S(p,q)$ is a $(p,q)$-tangle unknotting number one link, and that $S(vab+\epsilon, va^{2})$ is a $(\epsilon,v)$-tangle unknotting number one link. Our conjecture is supported when both summands are 2-bridge links.

In a previous paper \cite{bao}, we introduced a method of detecting whether a knot has H(2)-unknotting number one or not. The correction terms appearing in \cite{bao} are usually very challenging to calculate. When a knot $K$ is an alternating knot, there is a combinatorial formula for these correction terms. But in general, there is no practical rule to calculate them. Ozsv{\'a}th and Szab{\'o} \cite{MR2136532} used techniques related to plumbing manifolds and sharp manifolds to calculate these correction terms for some non-alternating knots. 

In Section~4, we want to apply the method in \cite{bao} to some tangle unknottig number one knots, which are usually non-alternating, without calculating their correction terms. Given a tangle unknotting number one knot $K$, our idea is to compare the correction terms of $K$ with those of certain 2-bridge knot, and to study $K$ via studying the 2-bridge knot. Note that 2-bridge knots are alternating and the correction terms for which are easy to calculate. A disadvantage of our method is it only works for some special tangle unknotting number one knots, rather than all of them. After introducing the theory, as an example, we show how to apply it to calculating the H(2)-unknotting numbers of some (23,3)-tangle unknotting number one knots.

\section{2-bridge links with H(2)-unknotting number one}

The 2-bridge links have been widely studied in knot theory. The 2-bridge link $S(p,q)$ to be used here is the link illustrated in Figure~\ref{fig:f3}. Here ${p}/{q}$ is the continued fraction $[a_{1}, a_{2},\cdots, a_{n}]$. Precisely $$\frac{p}{q}=[a_{1}, a_{2},\cdots, a_{n}]=a_{1}+\cfrac{1}{a_{2}+\cdots+\cfrac{1}{a_{n}}}.$$ Let $C(a_{1}, a_{2}, \ldots, a_{n})$ denote the link diagram in Figure~\ref{fig:f3}. Two unoriented links $K_{1}$ and $K_{2}$ are {\it equivalent} if there exists an orientation-preserving auto-homeomorphism of $S^{3}$ sending $K_{1}$ to $K_{2}$. The following fact is well-known: Two unoriented 2-bridge links $S(p_{1},q_{2})$ and $S(p_{2},q_{2})$ are equivalent if and only if $p_{1}=p_{2}$ and $q_{1}^{\pm}\equiv q_{2} \pmod{p_{1}}$.
\begin{figure}[h]
	\centering
		\includegraphics[width=0.7\textwidth]{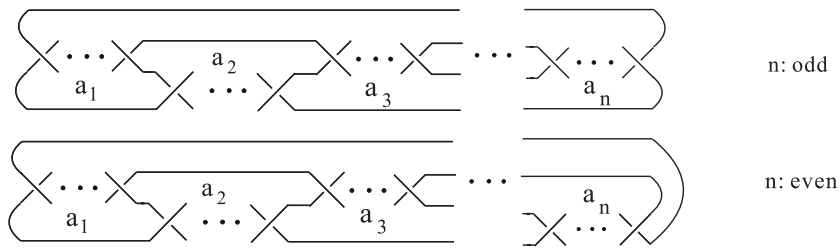}
	\caption{The link diagram $C(a_{1}, a_{2}, \ldots, a_{n})$ for 2-bridge links.}
	\label{fig:f3}
\end{figure}

The purpose of this section is to find out 2-bridge links whose H(2)-unknotting numbers are one. There is essentially no new theory here. We collect some known facts to get a necessary and sufficient condition for a 2-bridge link to have H(2)-unknotting number one, which is Proposition~\ref{th1}, and apply this observation to complement a table in \cite{MR2573402}, which we express as Corollary~\ref{table}. 

%A {\it Dehn surgery} on $S^{3}$ along a knot $K$ consists of the following two steps: first, drill out a tubular neighbourhood $nb(K)$ of the knot, which is a solid torus; second, glue back a solid torus $T$ by a homeomorphism $\varphi$ of its boundary to the torus boundary component of $S^{3}-nb(K)$. We can pick two oriented simple closed curves $m$ and $l$ on the boundary torus that generate the first homology group of the solid torus. This gives any simple closed curve $\gamma$ on the torus two coordinates p and q, each coordinate corresponding to the algebraic intersection of the curve with $m$ and $l$ respectively. These coordinates only depend on the homotopy class of $\gamma$. If the curve $m$ on $T$ is sent by $\varphi$ to a curve on $nb(K)$ with coordinate $(p,q)$, then we denote the resulting manifold by $S_{p/q}^{3}(K)$ and call it a $p/q$-surgery along $K$. For example the $p/q$-surgery along the unknot is the lens space $L(p,q)$.

We use $S_{r}^{3}(K)$ to denote the 3-manifold obtained by doing Dehn surgery to the 3-sphere $S^{3}$ along the knot $K$ with coefficient $r$. Our orientation convention is the $p/q$-surgery along the trivial knot gives the lens space $L(p,q)$. An oriented knot $C$ is called {\it strongly-invertible} if there is an orientation preserving homeomorphism, which is also an involution, of $S^{3}$ which takes the knot to itself but reverses the orientation along the knot. Given a link $K$, let $\Sigma(K)$ denote the double branched cover of $S^{3}$ along $K$. It is well-known that the double branched cover of $S^{3}$ along the 2-bridge link $S(p, q)$ is the lens space $L(p,q)$. We have the following lemma.

\begin{lemma}
\label{mon}
The 2-bridge link $S(p,q)$ has H(2)-unknotting number one if and only if the lens space $L(p,q)$ is $S_{\pm p}^{3}(C)$ for some strongly-invertible knot $C$.
\end{lemma}
\begin{proof}
The proof is in fact a practice of Montesinos' trick \cite{MR0380802}. In general, if a link $K$ has H(2)-unknotting number one, then $\Sigma(K)$ equals $S_{p}^{3}(C)$ for some strongly-invertible knot $C$ and $|p|$ equals the absolute value of the determinant of $K$. For 2-bridge links, the converse is true as well. The reason is as follows. For a integer $p$ and a strongly-invertible knot $C$, one can always construct a link with H(2)-unknotting number one for which the double branched cover is $S_{p}^{3}(C)$. The double branched cover of $S^{3}$ along the 2-bridge link $S(p, q)$ is the lens space $L(p,q)$, and it is known \cite{MR823282} that there is no other links sharing the same double branched cover with $S(p,q)$. Therefore, if the lens space $L(p,q)$ equals $S_{\pm p}^{3}(C)$, the H(2)-unknotting number of the 2-bridge link $S(p,q)$ must be one.
\end{proof}

Since the set of strongly-invertible knots is too large, Lemma~\ref{mon} does not help us simplify the task of finding out 2-bridge links with H(2)-unknotting number one. On the other hand, there have been many studies about integral surgeries which produce lens spaces. In the following paragraphs, we assemble some of these studies and come up with a practical criterion, which is Proposition~\ref{th1}, for distinguishing 2-bridge links with H(2)-unknotting number one.

If some integral surgery of $S^{3}$ along a knot gives rise to a lens space, we say this knot {\it admits integral lens space surgery}. It is known that doubly-primitive knots admit integral lens space surgeries. Here is the definition. Given a loop in the boundary of a genus two handlebody, it is called {\it primitive} if attaching a 2-handle produces a solid torus. A knot in $S^{3}$ is called {\it doubly-primitive} if it sits on a genus two Heegaard surface of $S^{3}$, and is primitive in handlebodies on both sides. Berge \cite{Berge} found twelve classes of doubly-primitive knots, which are now called Berge knots. The lens spaces which arise as integral surgeries along Berge knots are listed as follows in \cite{rasmussen}:
\begin{theo}[Berge]
\label{berge}
The lens space $L(\alpha,\beta)$ arises as an integral surgery along a Berge knot if there exists an integer $k$ such that $\beta\equiv \pm k^{2} \pmod{\alpha}$, and $\alpha, \beta$ and $k$ satisfy one of the following conditions:
\begin{enumerate}
	\item $\alpha \equiv ik \pm 1 \pmod{k^{2}}$ and $\gcd(i,k)=1,2$ for some $i$;
	\item $\alpha \equiv \pm (2k+\epsilon)d \pmod{k^{2}}$, $d|k-\epsilon$ and $\frac{k-\epsilon}{d}$ is odd, for $\epsilon=\pm 1$;
	\item $\alpha \equiv \pm (k+\epsilon)d \pmod{k^{2}}$ and $d|2k-\epsilon$, for $\epsilon=\pm 1$;
	\item $\alpha \equiv \pm (k+\epsilon)d \pmod{k^{2}}$, $d|k+\epsilon$ and $d$ is odd, for $\epsilon=\pm 1$;
	\item $k^{2}\pm k \pm 1 \equiv 0 \pmod{\alpha}$;
	\item $\alpha=22j^{2}+9j+1$ and $k=11j+2$ for some $j$;
	\item $\alpha=22j^{2}+13j+2$ and $k=11j+3$ for some $j$.
\end{enumerate}
\end{theo}

The following fact is containd, though not directly stated, in \cite{MR2366182}.

\begin{theo}
\label{ichi}
The 2-bridge link $S(p,q)$ has H(2)-unknotting number one when the lens space $L(p,q)$ can be obtained as an integral surgery along a doubly-primitive knot in $S^{3}$.
\end{theo}

Berge \cite{berge2} proved that every doubly-primitive knot in $S^{3}$ is a Berge knot. Greene \cite{joshua} Proved the following result:

\begin{theo}[\cite{joshua}]
\label{joshua}
If a lens space is realized as an integral surgery along a knot in $S^{3}$, then it can be realized as an integral surgery along some Berge knot. 
\end{theo}

The proof of Proposition~\ref{th1} now easily follows from Theorems~\ref{ichi} and \ref{joshua}. Lens spaces which arise as integral surgeries along Berge knots have been completely listed in Theorem~\ref{berge}. Therefore, the corresponding 2-bridge links are those 2-bridge links whose H(2)-unknotting numbers are one. To proof Corollary~\ref{table}, we only need to show that these four 2-bridge knots belong to the list.
%As a corollary, we may have:
%\begin{cor}
%The 2-bridge link $S(\alpha, \beta)$ has H(2)-unknotting number one if and only if the lens space $L(\alpha, \beta)$ can be obtained as an integral surgery along a knot in $S^{3}$.
%\end{cor}
%Ozsvath and Szabo established an effective characterization for lens spaces that are realized as integer surgeries along knots in $S^{3}$.
%They also checked that their charcterizaiton 

%In \cite{MR2573402}, an incomplete table of H(2)-unknotting number of knots is provided. Among knots with nine crossings, there are six knots whose H(2)-unknotting number are unknown. We confirm that the 2-bridge knots $9_{21}, 9_{23}, 9_{26}$ and $9_{31}$ are knots with H(2)-unknotting number one.

%\begin{cor}
%the 2-bridge knots $9_{21}, 9_{23}, 9_{26}$ and $9_{31}$ have H(2)-unknotting number one.
%\end{cor}
\begin{proof}[Proof of Corollary~\ref{table}]
Within this proof, we do not distinguish a knot from its mirror image. If the corresponding lens spaces belong to the list in Theorem~\ref{berge}, then Proposition~\ref{th1} tells us that the 2-bridge knots have H(2)-unknotting number one. 
For $9_{21}=S(43,25)$, we have $43=d(2k-1)\pmod{k^{2}}$ and $25=k^{2}$ for $k=5$ and $d=2$. It belongs to Berge type (ii).
For $9_{23}=S(45,64)$, we have $45=d(2k-1)\pmod{k^{2}}$ and $64=k^{2}$ for $k=8$ and $d=3$. It belongs to Berge type (ii).
For $9_{26}=S(47,81)$, we have $47=-d(2k-1)\pmod{k^{2}}$ and $81=k^{2}$ for $k=9$ and $d=2$. It belongs to Berge type (ii).
For $9_{31}=S(55,144)$, we have $55=d(k-1)\pmod{k^{2}}$ and $144=k^{2}$ for $k=12$ and $d=5$. It belongs to Berge type (iii).
\end{proof}

\section{Composite links with H(2)-unknotting number one}
%Composite links with H(2)-unknotting number one were studied by Bleiler \cite{MR1045396} and Eudave-Mu$\tilde{\rm n}$oz \cite{MR1112545}.
%Here is their result:

%\begin{theo}[\cite{MR1045396}, \cite{MR1112545}]
%If a composite link has H(2)-unknotting number one, then it is a connected sum of a 2-bridge link and a prime link.
%\end{theo}
In this section, we study composite links with H(2)-unknotting number one. We mainly focus on the proofs of Propositions~\ref{main} and \ref{main2}.
\begin{defn}
\label{def}
{\rm
A link $K$ is a $(p,q)$-tangle unknotting number one link if there is a tangle decomposition $K=T_{1}\cup T_{2}$ such that $T_{1}$ is the rational tangle as shown in Figure~\ref{fig:f4} and that $T_{0}\cup T_{2}$ is the unknot. Here $p/q$ is the continued fraction $[a_{1}, a_{2},\cdots, a_{n}]$. We call $T_{1}$ a rational tangle with Conway notation $(a_{1}, a_{2},\cdots, a_{n})$.}
\end{defn}
Note that the definition depends on not only $(p,q)$, but also the sequence of numbers $(a_{1}, a_{2},\cdots, a_{n})$. Our convention for Conway notation may be different from those in some references.
\begin{figure}[h]
	\centering
		\includegraphics[width=0.9\textwidth]{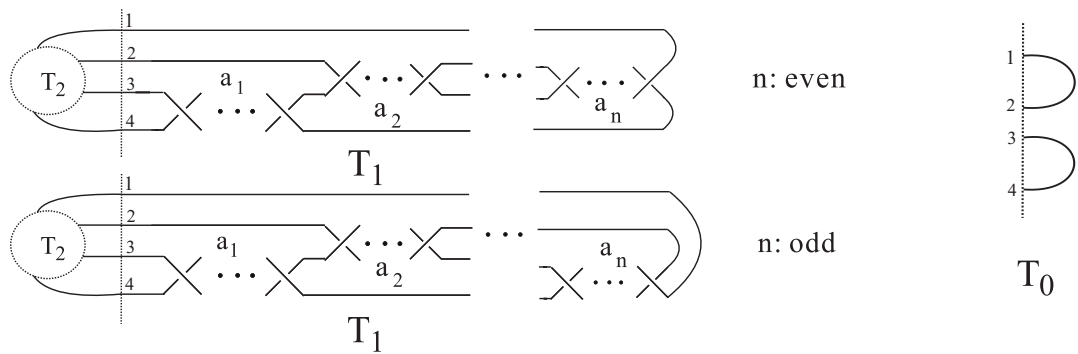}
	\caption{The tangle decomposition for a $(p,q)$-tangle unknotting number one knot.}
		\label{fig:f4}
\end{figure}

\begin{rem}
The notation ``tangle unknotting number one link" comes from the preprint \cite{Hedden}, but the definition here is slightly modified.
\end{rem}

\begin{lemma}
\label{lemma1}
If a link $K$ is a $(p,q)$-tangle unknotting number one link, then this link is also a $(p+aq,q)$-tangle unknotting number one link for any integer $a$. Furthermore, in this case there is an integer $a$ such that $\Sigma(K)=S_{(p+aq)/q}(C)$ for some strongly invertible knot $C$.
\end{lemma}
\begin{proof}
We first explain the former statement. Let $K=T_{2}\cup T_{1}$ be the tangle decomposition as in Definition~\ref{def} and $T_{1}$ be the rational tangle with notation $(a_{1}, a_{2},\cdots, a_{n})$. Given an integer $a$, let $T_{1}^{a}$ be the rational tangle with notation $(a_{1}+a, a_{2},\cdots, a_{n})$ and $T_{2}^{a}$ be the tangle obtained from $T_{2}$ by making $(-a)$ half twists along the endpoints $3$ and $4$. It is easy to see that $K=T_{1}^{a}\cup T_{2}^{a}$ and that $T_{2}^{a}\cup T_{0}$ is the unknot. Therefore $K$ is a $(p+aq,q)$-tangle unknotting number one link as well.

Now we prove the latter statement, which is again a practise of Montesinos' trick. Consider the tangle decomposition $$(S^{3}, T_{0}\cup T_{2})=(D^{3}, T_{0})\cup (D^{3}, T_{2}).$$
Since $T_{0}\cup T_{2}$ is the unknot, the double branched cover of $S^{3}$ along $T_{0}\cup T_{2}$ is still $S^{3}$. The double branched cover of $D^{3}$ along $T_{0}$ is a solid torus, which we denote $S_{0}$. Therefore, the manifold obtained by attaching $S_{0}$ to the double branched cover of $D^{3}$ along $T_{2}$ along their common torus boundary, is $S^{3}$. This implies that the double branched cover of $D^{3}$ along $T_{2}$ is the complement of a knot, say $C$, in $S^{3}$. In order to see that $C$ is strongly invertible, we notice that the image of $C$ in the base space $S^{3}$ is the dotted arc as shown in Figure~\ref{fig:f11}. It is easy to see that the preimage of this arc, which is $C$, is strongly invertible.

Next, we consider the tangle decomposition of $K$:
$(S^{3}, K)=(D^{3}, T_{1})\cup (D^{3}, T_{2})$.
Taking the double branched cover, we get
$$\Sigma(K)=S_{1}\cup (S^{3}\setminus \nu(C)),$$ where $S_{1}$ denotes the double branched cover of $D^{3}$ along $T_{1}$, which is a solid torus. Let $(m_{0}, l_{0})$, $(m_{1}, l_{1})$ and $(m,l)$ be the prefered meridian-longitudes of $S_{0}$, $S_{1}$ and $C$ respectively. Here we choose the same orientation for $S_{0}$ and $S_{1}$. We can see that $m_{1}=pm_{0}+ql_{0}$, while we already know that $m_{0}=m$ and $l_{0}=l+am$ for some integer $a$, so as a conclusion we have $m_{1}=(p+aq)m+ql$. That is to say $\Sigma(K)=S^{3}_{(p+aq)/q}(C)$. 
\end{proof}

\begin{figure}
	\centering
		\includegraphics[width=0.8\textwidth]{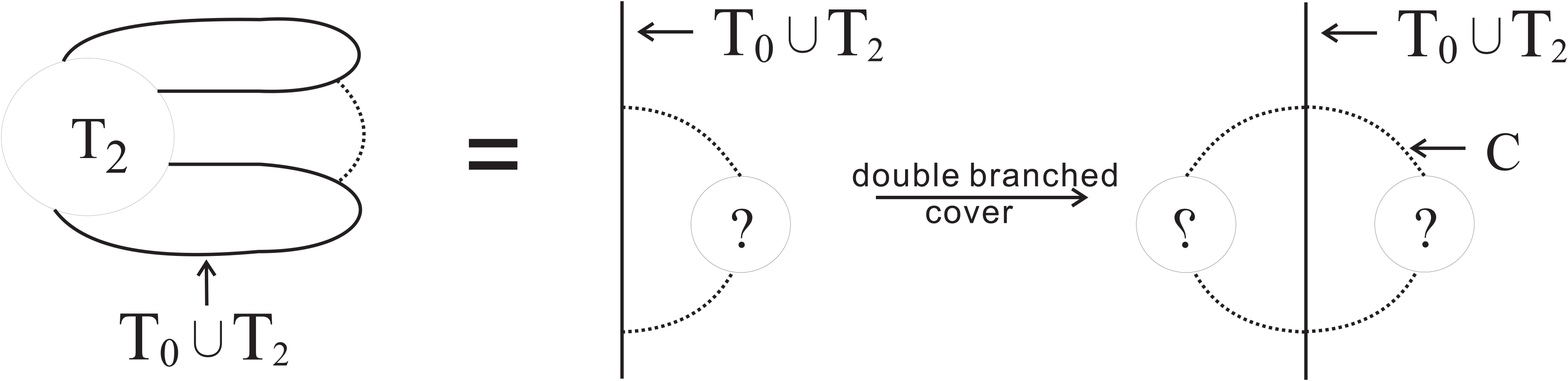}
		\caption{The knot $C$.}
	\label{fig:f11}
\end{figure}

%\begin{prop}
%\label{main}
%For the 2-bridge link $S(q, p)$, and a $(p,q)$-tangle unknotting number one link $K(p, q)$, the H(2)-unknotting number of $S(q, p)\sharp K(p, q)$ is one.
%\end{prop}
\begin{proof}[Proof of Proposition~\ref{main}]
In fact, if $p/q=[a_{1}, a_{2},\cdots, a_{n}]$, then $S(q,p)$ is equivalent to $S(q,p-a_{1}q)$. Note that $q/(p-a_{1}q)=[a_{2},\cdots, a_{n}]$. Then  $C(a_{2},\cdots, a_{n})$ is a link diagram for $S(q,p)$. The composite link $K(p, q)\sharp S(q, p)$ can be unknotted by adding a band, as shown in Figure~\ref{fig:f5}.
\begin{figure}[h]
	\centering
		\includegraphics[width=1.0\textwidth]{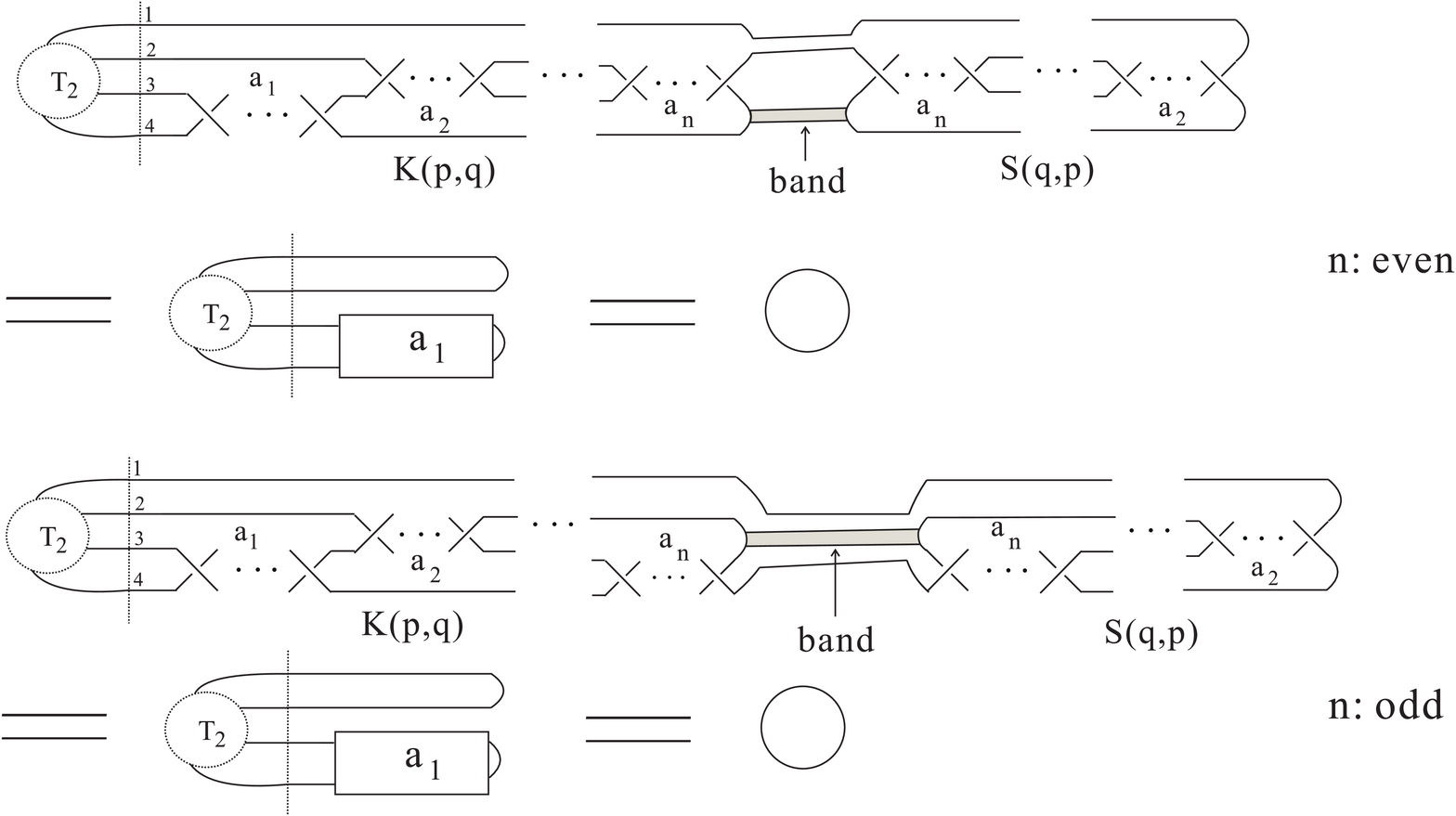}
	\caption{The H(2)-unknotting number of the composite link is one.}
		\label{fig:f5}
\end{figure}
This completes the proof.
\end{proof}

%\begin{proof}[Second proof of Proposition~\ref{main}]
%Since $K(p,q)$ is a $(p,q)$-tangle unknotting number one knot, by Lemma~\ref{lemma1} we have $\Sigma(K(p,q))=S_{(p+aq)/q}^{3}(C)$ for some strongly invertible knot $C$ and some integer $a$.
%Consider the $(p+aq,q)$-cable knot of $C$, which we denote $C_{p+aq,q}$. Then $C_{p+aq,q}$ is a strongly invertible knot as well. It is known that \cite{}
%\begin{eqnarray*}
%S_{(p+aq)q}^{3}(C_{p+aq,q})=L(q,p+aq)\sharp S_{(p+aq)/q}^{3}(C)&=&\Sigma(S(q,p+aq))\sharp \Sigma(K(p,q))\\
%&=&\Sigma(S(q,p+aq)\sharp K(p,q)).
%\end{eqnarray*}
%Note that $S(q,p+aq)=S(q,p)$.
%We see that $u_{2}(S(q,p)\sharp K(p,q))=1$.
%\end{proof}

We conjecture that the converse of Proposition~\ref{main} is true:
\begin{con}
\label{q1}
A composite link with H(2)-unknotting number one always has the form described in Proposition~\ref{main2}.
\end{con}

%Conjecture~\ref{q1} holds if we have a positive answer to the Question~\ref{ques}. We will give the explanation after the proof of %Proposition~\ref{main}.
%\begin{que}
%\label{ques}
%Is it true that a knot $K$ is strongly invertible if and only if certain cable of K is strongly invertible?
%\end{que}

We have the following corollary since the 2-bridge link $S(p,q)$ is a $(p,q)$-tangle unknotting number one link. It is in fact Theorem~9.1 in \cite{MR2573402}.
\begin{cor}
\label{cor}
The H(2)-unknotting number of the link $S(p, q)\sharp S(q, p)$ is one.
\end{cor}

%\begin{prop}
%\label{main2}
%Suppose $S(p,q)$ and $S(r,s)$ are two non-trivial 2-bridge links. Then the composite $S(p,q)\sharp S(r,s)$ has H(2)-unknotting number one if and only if either $S(r,s)=S(q,p)$, or $S(p,q)=S(v,\epsilon)$ and $S(r,s)=S(vab+\epsilon, va^{2})$ for $\epsilon = \pm 1$ and some integers $v$,$a$ and $b$ satisfying $(a,b)=1$.
%\end{prop}
Before giving the proof of Proposition~\ref{main2}, we introduce some facts. Given a knot $K$ and two coprime integers $p$ and $q$, we use $K_{p,q}$ to denote the $(p,q)$-cable knot of $K$. It is conjectured that (the cabling conjecture \cite{MR809502}) that if an integral surgery of $S^{3}$ along a knot $L$ produces a reducible manifold, then $L$ is a cable knot and the slope of the surgery is $pq$. This conjecture holds when $L$ is a strongly invertible knot \cite{MR1112545}.

\begin{proof}[Proof of Proposition~\ref{main2}]
If $S(p,q)\sharp S(r,s)$ has H(2)-unknotting number one, then there exists a strongly invertible knot $C$ and an integer $l$ such that $\Sigma(S(p,q)\sharp S(r,s))=L(p,q)\sharp L(r,s)=S^{3}_{l}(C)$. Since the cabling conjecture holds for strongly invertible knots, the knot $C$ must be a cable knot, say $C=K_{u,v}$ for some knot $K$ and coprime integers $u$ and $v$. Then we have $l=uv$ and $S^{3}_{l}(C)=S_{u/v}^{3}(K)\sharp L(v,u)$, which in turn equals $L(p,q)\sharp L(r,s)$. By the prime decomposition theorem for 3-manifolds, we can suppose $L(v,u)=L(p,q)$. Then $S_{u/v}^{3}(K)$ has to be $L(r,s)$. The cyclic surgery theorem \cite{MR881270} implies that if a non-integral Dehn surgery of $S^{3}$ along a knot produces a Lens space, then the knot is a torus knot. Since $|v|=|p|>1$, the fact $S_{u/v}^{3}(K)=L(r,s)$ implies that $K$ must be a torus knot.

If $K$ is the unknot, then $S_{u/v}^{3}(K)=L(u,v)=L(r,s)$. Therefore $S(r,s)$ is equivalent to $S(q,p)$. (In fact, $S(r,s)$ may be $S(q+jp,p)$ for some integer $j$, but in this case, we can write $S(p,q)$ as $S(p, q+jp)$.) If $K$ is non-trivial, suppose $K$ is the $(a,b)$-torus knot. Then $S_{u/v}^{3}(K)$ is a lens space only if $u=vab\pm 1$, and then $S_{u/v}^{3}(K)=L(vab\pm 1, va^{2})$. In this case, $S(p,q)=S(v,vab\pm 1)=S(v,\pm 1)$ and $S(r,s)=S(vab\pm 1, va^{2})$. 

The converse can be proved easily.
\end{proof}

%Now we go back to Conjecture~\ref{q1}. Given two non-trivial knots $K_{1}$ and $K_{2}$, if the H(2)-unknotting number of $K_{1}\sharp K_{2}$ is one, then we have $\Sigma(K_{1}\sharp K_{2})=\Sigma(K_{1})\sharp \Sigma(K_{2})=S_{p}(C)$ for some strongly invertible knot $C$. As we mentioned before, the cabling conjecture holds for strongly invertible knots. The knot $C$ must be a cable knot, say $C=K_{u,v}$ for some knot $K$ and coprime integers $u$ and $v$, and $p$ must be $uv$. If We have a confirmative answer to Question~\ref{ques}, then $K$ is a strongly invertible knot. Then $\Sigma(K_{1})\sharp \Sigma(K_{2})=S_{p}(C)=S_{u/v}^{3}(K)\sharp L(v,u)$. Suppose $\Sigma(K_{1})=S_{u/v}^{3}(K)$ and $\Sigma(K_{2}=L(v,u)$.

From Propositions~\ref{main} and \ref{main2}, we have the following corollary:
\begin{cor}
\label{cor1}
\begin{enumerate}
	\item The H(2)-unknotting number of a $(p,q)$-tangle unknotting number one link is less than or equal to $u_{2}(S(q,p))+1$.
	\item $u_{2}(C(a_{1},a_{2}, \cdots,a_{n}))\leq u_{2}(C(a_{i},a_{i+1}, \cdots,a_{n}))+i-1$.
\end{enumerate}

\end{cor}

\begin{rem}
The 2-bridge link $S(vab+\epsilon, va^{2})$ in Proposition~\ref{main2} is an $(\epsilon, v)$-tangle unknotting number one link.
\end{rem}
%the uniqueness of the doubly-primitive knots,
%positive surgery and negative surgery.
%uniqueness of bands unknotting the knot.
\section{Tangle unknotting number one knots}
In Section~4.1, we recall some facts in Heegaard Floer homology for our discussions in Sections 4.2 and 4.3. In Section 4.2 we establish Relation (\ref{3}), which is the central result in Section~4. In Section 4.3 we apply it to calculating the H(2)-unknotting numbers of some tangle unknotting number one knots. 
\subsection{Preliminaries}
Almost all the ingredients contained in this subsection can be found in \cite{MR2136532}, or an earlier paper \cite{MR1957829}. But for intactness, we include them here. If $X$ is an oriented 3- or 4-manifold, the second cohomology $H^{2}(X,{\mathbb Z})$ acts on the set of spin$^{c}$-structures ${\rm Spin}^{c}(X)$ freely and transitively. Each spin$^{c}$-structure $s\in {\rm Spin}^{c}(X)$ has the first Chern class $c_{1}(s)\in H^{2}(X,{\mathbb Z})$, and the relation to the action is $c_{1}(s+h)=c_{1}(s)+2h$ for any $h \in H^{2}(X,{\mathbb Z})$. 

Let $Y$ be an oriented rational homology 3-sphere and $s$ be a spin$^{c}$-structure over $Y$. Then there is Heegaard Floer homology associated with the pair $(Y, s)$. In this note, we use Heegaard Floer homology with coefficients in the field ${\mathbb F}:={\mathbb Z}/2{\mathbb Z}$. There are several versions of this homology. One version is $HF^{+}(Y,s)$, which is a ${\mathbb Q}$-graded module over the polynomial algebra ${\mathbb F}[U]$. That is $$HF^{+}(Y, s)=\bigoplus_{i\in {\mathbb Q}}HF^{+}_{i}(Y,s),$$ where multiplication by $U$ lowers the grading by two. In each grading $i\in {\mathbb Q}$, $HF^{+}_{i}(Y,s)$ is a finite-dimensional ${\mathbb F}$-vector space. A simpler version is $HF^{\infty}(Y)$, and it satisfies $HF^{\infty}(Y,s)={\mathbb F}[U, U^{-1}]$ for each $s\in {\rm Spin}^{c}(Y)$ \cite[Theorem 10.1]{MR2113020}. It is also ${\mathbb Q}$-graded and multiplication by $U$ lowers its grading by two. 

There is a natural ${\mathbb F}[U]$-equivariant map $$\pi: HF^{\infty}_{i}(Y,s) \rightarrow HF^{+}_{i}(Y,s),$$ which is zero in all sufficiently negative gradings and an isomorphism in all sufficiently positive gradings. Note that $\pi$ preserves the $\mathbb Q$-grading. The map $\pi$ determines an invariant $d(Y,s)$, which is called the \textit{correction term} of the pair $(Y,s)$. Precisely $d(Y,s)$ is the minimal ${\mathbb Q}$-grading on which the map $\pi$ is non-zero. The correction terms for $Y$ and $-Y$, where $``-"$ means the reversion of orientation, are related by the formula $$d(-Y,s)=-d(Y,s)$$ under the natural identification ${\rm Spin}^{c}(Y)\cong {\rm Spin}^{c}(-Y)$.

The map $\pi$ behaves naturally under cobordisms. Let $Y_{1}$ and $Y_{2}$ be two oriented rational homology 3-spheres. We say a smooth connected oriented 4-manifold $X$ is a cobordism from $Y_{1}$ to $Y_{2}$ if the boundary of $X$ is given by $\partial X=(-Y_{1})\cup Y_{2}$. Suppose $X$ is a cobordism from $Y_{1}$ to $Y_{2}$ and $t$ is a spin$^{c}$-structure of $X$. Then there is a homomorphism $$F_{X,t}^{o}: HF^{o}(Y_{1}, s_{1})\rightarrow HF^{o}(Y_{2}, s_{2}),$$ where $HF^{o}$ denotes $HF^{+}$ or $HF^{\infty}$ and $s_{i}$ is the restriction of $t$ to $Y_{i}$ for $i=1,2$ (we simply express it as $s_{i}=t\bigm|_{Y_{i}}$). The map $\pi$ and the map $F_{X,t}^{o}$ fit into the following commutative diagram:
\begin{equation}
\label{eq3}
\begin{CD}
HF^{\infty}(Y_{1},s_{1}) @> F_{X,t}^{\infty} >> HF^{\infty}(Y_{2},s_{2})  \\
@V\pi_{1}VV @VV\pi_{2}V \\
HF^{+}(Y_{1},s_{1}) @> F_{X,t}^{+} >> HF^{+}(Y_{2},s_{2}).
\end{CD}
\end{equation}

When $X$ is a negative-definite 4-manifold it is shown in \cite{MR1957829} that 
\begin{eqnarray}
d(Y_{2}, t\bigm|_{Y_{2}})-d(Y_{1}, t\bigm|_{Y_{1}}) &\geq & \frac{c_{1}^{2}(t)-2\chi (W)-3\sigma (W)}{4},\label{1}\\
d(Y_{2}, t\bigm|_{Y_{2}})-d(Y_{1}, t\bigm|_{Y_{1}}) &=& \frac{c_{1}^{2}(t)-2\chi (W)-3\sigma (W)}{4} \pmod{2},\label{2}
\end{eqnarray}
where $\chi (W)$ is the Euler characteristic of $W$ and $\sigma (W)$ is the signature of $W$.
Both relations follow from the proof of \cite[Theorem 9.6]{MR1957829}, but they are not clearly stated. For readers' convenience, we explain them here. 
If $X$ is a negative-definite cobordism, the proof of Theorem 9.1 in \cite{MR1957829} (also mentioned in the proof of \cite[Proposition 9.9]{MR1957829}) tells us that $F_{X,t}^{\infty}$ is an isomorphism. There is an element $\xi\in HF^{\infty}(Y_{2}, t\bigm|_{Y_{2}})$ with the property that its $\mathbb Q$-grading $\operatorname{gr}(\xi)$ is $d(Y_{2}, t\bigm|_{Y_{2}})$. Suppose the preimage of $\xi$ in $HF^{\infty}(Y_{1}, t\bigm|_{Y_{1}})$ is $\eta$. Then by Equation (4) in \cite{MR1957829}, we have $$\operatorname{gr}(\xi)-\operatorname{gr}(\eta)=\frac{c_{1}^{2}(t)-2\chi (W)-3\sigma (W)}{4}=d(Y_{2},t\bigm|_{Y_{2}})-\operatorname{gr}(\eta).$$ 
By the definition of correction term, it is easy to see that $$\operatorname{gr}(\eta)\geq d(Y_{1},t\bigm|_{Y_{1}}).$$ Since $Y_{1}$ is an oriented rational homology 3-sphere, as an ${\mathbb F}$-vector space, we have (\cite[Theorem 10.1]{MR2113020}) $HF^{\infty}(Y_{1}, t\bigm|_{Y_{1}})=\bigoplus_{i=-\infty}^{\infty}{\mathbb F}_{(d+2i)}$, where $d=d(Y_{1},t\bigm|_{Y_{1}})$ and ${\mathbb F}_{(j)}$ denotes the summand supported on grading $j$. Therefore we have $$\operatorname{gr}(\eta)- d(Y_{1},t\bigm|_{Y_{1}})=0 \pmod{2}.$$ Now (\ref{1}) and (\ref{2}) follow from the argument above.

\subsection{Theory}
The purpose of this subsection is to prove Relation (\ref{3}), which will be applied in next subsection to calculate the H(2)-unknotting numbers of some tangle unknotting number one knots.
For a connected oriented rational homology 3-sphere $Y$, if the order of $H^{2}(Y, {\mathbb Z})$ is odd, there exists a group structure on the set ${\rm Spin}^{c}(Y)$ by identifying $s\in {\rm Spin}^{c}(Y)$ with $c_{1}(s)\in H^{2}(Y,{\mathbb Z})$. In this case, we also denote the correction term $d(Y,s)$ by $d(Y, c_{1}(s))$ if necessary. We have the following result about H(2)-unknotting number. We remark that the statement is modified slightly from the main theorem in \cite{bao}, but the correctness can be read out easily from the context.

\begin{theo}[\cite{bao}]
\label{bao}
Let $K$ be a knot and $p$ be the absolute value of the determinant of $K$. If $u_{2}(K)=1$, then there is a group isomorphism $\phi : {\mathbb Z}/p{\mathbb Z}\longrightarrow H^{2}(\Sigma(K); {\mathbb Z})$ and a sign $\epsilon =\pm 1$ with the properties that for all $i\in {\mathbb Z}/p{\mathbb Z}$:
\begin{eqnarray*}
I_{\phi, \epsilon}(i):=\epsilon\cdot d(\Sigma(K), \phi(i))+\frac{1}{4}(\frac{1}{p}(\frac{p+(-1)^{i}p}{2}-i)^{2}-1)&=&0 \pmod{2},\\
\text{and}\quad I_{\phi, \epsilon}(i) &\geq & 0.
\end{eqnarray*}
\end{theo}

%The correction terms are usually very challenging to calculate. When the knot $K$ is an alternating knot, there is a combinatorial formula for the correction terms of $\Sigma(K)$. But in general, there is no practice rule to calculate them. Ozsv{\'a}th and Szab{\'o} \cite{MR2136532} used techniques related to plumbing manifolds and sharp manifolds to calculate the correction terms of the double branched covers of some non-alternating knots. 

%In this section, we want to apply Theorem~\ref{bao} to some tangle unknottig number one knots, which are usually non-alternating, without calculating the correction terms. Given a $(p,q)$-tangle unknotting number one knot $K$, our idea is to compare the correction terms of $\Sigma(K)$ with those of $L(p,q)$, and to study $K$ through the 2-bridge knot $S(p,q)$. A disadvantage of our method in this note is it only works for some tangle unknotting one knots, rather than all of them.	

%The following is a general property of correction terms under cobodism.																																																																

%\begin{lemma}[Ozsv{\'a}th-Szab{\'o} \cite{MR1957829}]
%\label{os}
%Suppose $(Y_{1}, t_{1})$ and $(Y_{2}, t_{2})$ are two rational homology 3-sphere, and $(W, s)$ is a negative-definite cobordism from $(Y_{1}, t_{1})$ to $(Y_{2}, t_{2})$. Then the following hold:
%\begin{eqnarray*}
%d(Y_{2}, t_{2})-d(Y_{1}, t_{1})&\geq & \frac{c_{1}(s)^{2}-2\chi(W)-3\sigma(W)}{4}\quad \text{and}\\
%d(Y_{2}, t_{2})-d(Y_{1}, t_{1})&=& \frac{c_{1}(s)^{2}-2\chi(W)-3\sigma(W)}{4} \pmod{2}.
%\end{eqnarray*}
%\end{lemma} 
\begin{figure}[h]
	\centering
		\includegraphics[width=0.9\textwidth]{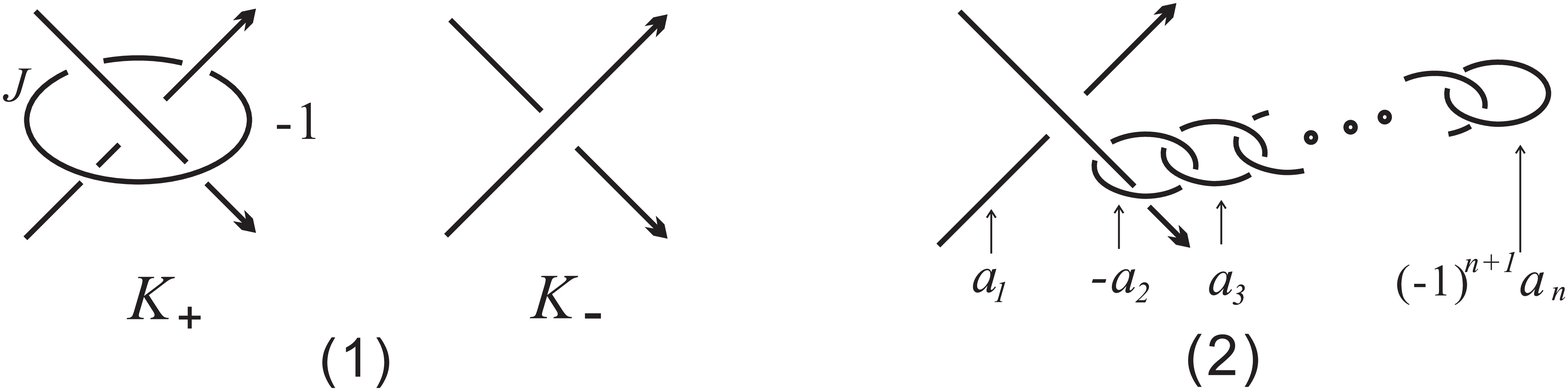}
		\caption{Kirby diagrams.}
	\label{fig:f7}
\end{figure}

Suppose $K_{+}$ and $K_{-}$ are two knots which differ only in a local neighborhood of a crossing, as shown in Figure~\ref{fig:f7}-(1) (Ignore the circle $J$ around the crossing). Consider the two manifolds $Y_{+}=S^{3}_{p/q}(K_{+})$ and $Y_{-}=S^{3}_{p/q}(K_{-})$, where $p$ is odd and $(p,q)=1$. There is a cobordism from $Y_{+}$ to $Y_{-}$ given by attaching a 2-handle to $Y_{+}\times [0,1]$ along the circle $J$ with framing $-1$, and we denote the cobordism by $W:Y_{+}\rightarrow Y_{-}$. Then we have the following property:
\begin{prop}
\label{pro}
The cobordism $W$ is a negative-definite cobordism from $Y_{+}$ to $Y_{-}$, and therefore:
\begin{eqnarray*}
d(Y_{-}, t_{-})-d(Y_{+}, t_{+}) &\geq & \frac{c_{1}^{2}(s)-2\chi (W)-3\sigma (W)}{4} \quad \text{and}\\
d(Y_{-}, t_{-})-d(Y_{+}, t_{+})&=& \frac{c_{1}^{2}(s)-2\chi (W)-3\sigma (W)}{4} \pmod{2},
\end{eqnarray*}
where $s$ is a spin$^{c}$-structure over $W$ with $s\bigm|_{Y_{*}}=t_{*}$ for $*=+,-$.
\end{prop}
\begin{proof}
First we prove $H_{2}(W; {\mathbb Z})={\mathbb Z}$ by the following argument. Let $W_{+}$ be the plumbing manifold obtained by attaching 2-handles to a four-ball along the framed link in Figure~\ref{fig:f7}-(2). The framings come from that $p/q$ equals the continued fraction $[a_{1}, a_{2}, \cdots, a_{n}]$. By the slam-dunk move in Figure~\ref{fig:f8}, we see the boundary of $W_{+}$ is indeed $Y_{+}$. Then we let $X=W_{+}\cup W$, and clearly we have $H_{2}(X; {\mathbb Z})={\mathbb Z}^{n+1}$ and $H_{2}(W_{+}; {\mathbb Z})={\mathbb Z}^{n}$. Consider the Mayer-Vietoris sequence applied to the decomposition $X=W_{+}\cup W$. In this case $W_{+}\cap W=S^{3}_{p/q}(K_{+})=Y_{+}$, so we have the following sequence for integral homology groups:
\begin{eqnarray*}
\cdots\rightarrow H_{2}(Y_{+})\rightarrow H_{2}(W_{+})\oplus H_{2}(W)\rightarrow H_{2}(X)\rightarrow H_{1}(Y_{+})\rightarrow\cdots.
\end{eqnarray*}
Since $ H_{2}(Y_{+})=0$, $H_{2}(W_{+})={\mathbb Z}^{n}$, $H_{2}(X)={\mathbb Z}^{n+1}$ and $H_{1}(Y_{+})={\mathbb Z}/p{\mathbb Z}$, we have $H_{2}(W)={\mathbb Z}$. 

It is easy to see that there exists a Seifert surface of $J$ in $S^{3}$ which is disjoint with $K_{+}$. Let $\alpha \in H_{2}(W; {\mathbb Z})$ be the homology class of this Seifert surface of $J$ capped off by the core disk of the two-handle attached along $J$. Then $\alpha$ is a generator of $H_{2}(W; {\mathbb Z})$, and we have $\alpha^{2}=-1$, which implies that the cobordism $W$ is negative-definite. By (\ref{1}) and (\ref{2}), we complete the proof of the proposition.
\end{proof}

\begin{rem}
The idea of the proof comes from the proof of Theorem 1.3 in \cite{thomas}.
\end{rem}

\begin{figure}
	\centering
		\includegraphics[width=0.5\textwidth]{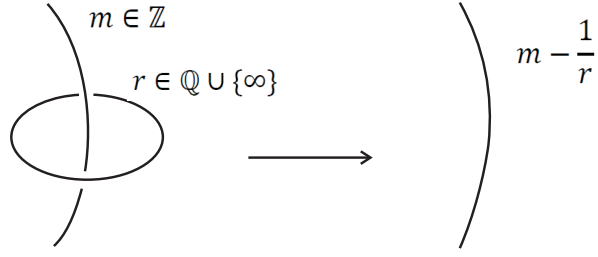}
		\caption{The slam-dunk move.}
	\label{fig:f8}
\end{figure}

In the following two paragraphs, we figure out when the pair of spin$^{c}$-structures $(t_{+}, t_{-})\in \operatorname{Spin}^{c}(Y_{+})\times \operatorname{Spin}^{c}(Y_{-})$ can be realized as the restriction of a spin$^{c}$-structure over $W$.

Note that $W$ is constructed by attaching a 2-handle $D^{4}=D^{2}\times D^{2}$ to $Y_{+}\times [0,1]$ along a solid torus $S^{1}\times D^{2}=(\partial D^{2})\times D^{2}$.
By considering the Mayer-Vietoris Sequence associated with the triple $(S^{1}\times D^{2}, Y_{+}\times [0,1]\coprod D^{4}, W)$, we have
$$0\rightarrow H_{2}(W)\rightarrow H_{1}(S^{1}\times D^{2})={\mathbb Z}\stackrel{f}{\rightarrow} H_{1}(Y_{+}\times [0,1])\oplus H_{1}(D^{4})={\mathbb Z}/p{\mathbb Z} \rightarrow H_{1}(W)\rightarrow 0.$$ It is easy to see that $H_{1}(S^{1}\times D^{2})$ is generated by a longitude of $J$ in Figure~\ref{fig:f7}, whose homology class in $H_{1}(Y_{+}\times [0,1])$ is zero. Therefore the map $f$, which is induced by the inclusion map, is trivial. Therefore we have $H_{1}(W)={\mathbb Z}/p{\mathbb Z}$. By the universal coefficient theorem, we have $H^{2}(W)={\mathbb Z}\oplus {\mathbb Z}/p{\mathbb Z}$.  
Noting that $\partial W=(-Y_{+})\cup Y_{-}$, we have the following exact sequence with respect to the pair $(W, \partial W)$:
\begin{multline*}
0\rightarrow H_{2}(W)={\mathbb Z}\stackrel{\tau}{\rightarrow}H^{2}(W)={\mathbb Z}\oplus {\mathbb Z}/p{\mathbb Z}\stackrel{\alpha}{\rightarrow}H^{2}(-Y_{+})\oplus H^{2}(Y_{-})={\mathbb Z}/p{\mathbb Z}\oplus {\mathbb Z}/p{\mathbb Z}\\ \stackrel{\beta}{\rightarrow} H_{1}(W)={\mathbb Z}/p{\mathbb Z}\rightarrow 0.
\end{multline*}
Let $m_{+}\subset Y_{+}$ be a meridian of $K_{+}$ and $m_{-}\subset Y_{-}$ be the image of $m_{+}$ after the Dehn surgery along $J$. Then $[m_{+}]$ and $[m_{-}]$ are generators of $H_{1}(Y_{+})$ and $H_{1}(Y_{-})$ respectively. We identify $H^{2}(Y_{+})$ with $H_{1}(Y_{+})$ and $H^{2}(Y_{-})$ with $H_{1}(Y_{-})$ by Poincar{\'e} duality. Then we have $\beta(-[m_{+}], [m_{-}])=0 \in H_{1}(W)$.

The set Spin$^{c}$(W) is an affine space over $H^{2}(W; {\mathbb Z})={\mathbb Z}\oplus {\mathbb Z}/p{\mathbb Z}$. Given a pair of spin$^{c}$-structures $t_{+}\in {\rm Spin}^{c}(Y_{+})$ and $t_{-}\in {\rm Spin}^{c}(Y_{-})$, a spin$^{c}$-structure $s\in {\rm Spin}^{c}(W)$ satisfies $s\bigm|_{Y_{*}}=t_{*}$ for $*=+,-$ if and only if $\alpha(c_{1}(s))=(-c_{1}(t_{+}), c_{1}(t_{-}))$. From the exactness of the sequence, the element $(-i,i):=(-i[m_{+}], i[m_{-}])\in H^{2}(-Y_{+})\oplus H^{2}(Y_{-})$ stays in the image of the map $\alpha$ for $0\leq i \leq p-1$. Let $t_{+}\in {\rm Spin}^{c}(Y_{+})$ and $t_{-}\in {\rm Spin}^{c}(Y_{-})$ be the spin$^{c}$-structures for which $(-c_{1}(t_{+}), c_{1}(t_{-}))=(-i,i)$. Then $(t_{+}, t_{-})$ is the restriction of some spin$^{c}$-structures over $W$ to $(Y_{+}, Y_{-})$. 

We use $d(Y_{-}, i)$ (resp. $d(Y_{+}, i)$) to denote $d(Y_{-}, t_{-})$ (resp. $d(Y_{+}, t_{+})$). In the following, we want to show that 
\begin{equation}
\label{3}
\begin{array}{l}
d(Y_{-}, i)-d(Y_{+}, i)\geq 0\quad \text{and}\\
d(Y_{-}, i)-d(Y_{+}, i)=0 \pmod{2},
\end{array}
\end{equation}
for any $i\in {\mathbb Z}/p{\mathbb Z}$. By (\ref{1}) and (\ref{2}), we have the following relations:
\begin{equation*}
\begin{array}{l}
d(Y_{-}, i)-d(Y_{+}, i)\geq  \frac{c_{1}^{2}(s)-2\chi(W)-3\sigma(W)}{4}\quad \text{and}\\
d(Y_{-}, i)-d(Y_{+}, i)= \frac{c_{1}^{2}(s)-2\chi(W)-3\sigma(W)}{4} \pmod{2},
\end{array}
\end{equation*}
for any spin$^{c}$-structure $s$ over $W$ whose restriction to $(Y_{+}, Y_{-})$ is $(t_{+}, t_{-})$. Note that $\chi(W)=1$ and $\sigma(W)=-1$. Then we can prove (\ref{3}) if we can prove that \begin{equation}\label{4}\max\left\{c_{1}^{2}(s)\left|s\in \operatorname{Spin}^{c}(W), s\bigm|_{Y_{*}}=t_{*} \text{ for } *=+,-\right.\right\}=-1.\end{equation}

We say $\xi\in {\mathbb Z}$ is a {\it characteristic element} of the matrix $(-1)$ if $\xi$ is odd. Let $t_{+,0}$ and $t_{-,0}$ denote the spin$^{c}$-structures whose first Chern classes are trivial, over $Y_{+}$ and $Y_{-}$ respectively. We define a set $H:=\left\{s\in \operatorname{Spin}^{c}(W)\left| s\bigm|_{Y_{*}}=t_{*,0} \text{ for } *=+,-\right.\right\}$. Then the first Chern classes of elements in $H$ belong to the free part of $H^{2}(W; {\mathbb Z})={\mathbb Z}\oplus {\mathbb Z}/p{\mathbb Z}$. Conversely any spin$^{c}$-structure over $W$ whose first Chern class belongs to the free part of $H^{2}(W; {\mathbb Z})$, belongs to $H$. The set $\left\{c_{1}(s)\left|s\in H \right.\right\}$ is equal to the set of characteristic elements of the matrix $(-1)$. If $c_{1}(s)$ corresponds to the characteristic element $\xi$, then we have $c_{1}^{2}(s)=-\xi^{2}$. It is easy to calculate that $\max\left\{c_{1}^{2}(s)\left|s\in H\right.\right\}$ is $-1$. 

Note that $H^{2}(W; {\mathbb Z})$ acts transtively and freely on the set $\operatorname{Spin}^{c}(W)$. Any spin$^{c}$-structure over $W$ can be transformed into a spin$^{c}$-structure in $H$ by a torsion element of $H^{2}(W; {\mathbb Z})$. We know that for a spin$^{c}$-structure $s$ and a torsion element $a\in H^{2}(W; {\mathbb Z})={\mathbb Z}\oplus {\mathbb Z}/p{\mathbb Z}$ there is $c_{1}(s+a)^{2}=c_{1}(s)^{2}$.  Therefore Relation (\ref{4}) is true, which in turn implies (\ref{3}).
 
\subsection{Application}

We show that (\ref{3}) can be used to study the H(2)-unknotting numbers of some tangle unknotting number one knots. Let us consider some (23,3)-tangle unknotting number one knots as examples. For the 2-bridge knot $S(23,3)$, the ordered set of the correction terms of its double-branched cover, the lens space $L(23,3)$, is given as follows (for the calculation, refer to \cite[Proposition 3.2]{MR2136532}). Here we identify Spin$^{c}$(L(23,3)) with ${\mathbb Z}/23{\mathbb Z}$.
\begin{eqnarray*}
\{d(L(23,3),i)\}_{i=0}^{22}=\left\{ \frac{3}{2}, \frac{85}{46}, \frac{41}{46}, \frac{29}{46}, \frac{49}{46}, \frac{9}{46}, \frac{1}{46}, \frac{25}{46}, \frac{-11}{46}, \frac{-15}{46}, \frac{13}{46}, \frac{-19}{46}, \frac{-19}{46}\right.,\\
\left.\frac{13}{46}, \frac{-15}{46}, \frac{-11}{46}, \frac{25}{46}, \frac{1}{46}, \frac{9}{46}, \frac{49}{46}, \frac{29}{46}, \frac{41}{46}, \frac{85}{46}\right\}.
\end{eqnarray*}
On the other hand, define $f(i)=\frac{1}{4}(\frac{1}{p}(\frac{p+(-1)^{i}p}{2}-i)^{2}-1)$ for $p=23$ and any $0 \leq i \leq 22$. Then we have the following ordered set:
\begin{eqnarray*}
\{f(i)\}_{i=0}^{22}=\left\{ \frac{11}{2}, \frac{-11}{46}, \frac{209}{46}, \frac{-7}{46}, \frac{169}{46}, \frac{1}{46}, \frac{133}{46}, \frac{13}{46}, \frac{101}{46}, \frac{29}{46}, \frac{73}{46}, \frac{49}{46}, \frac{49}{46}\right.,\\
\left.\frac{73}{46}, \frac{29}{46}, \frac{101}{46}, \frac{13}{46}, \frac{133}{46}, \frac{1}{46}, \frac{169}{46}, \frac{-7}{46}, \frac{209}{46}, \frac{-11}{46}\right\}.
\end{eqnarray*}

We will apply Theorem~\ref{bao} to show that $u_{2}(S(23,3))>1$. Assume that $u_{2}(S(23,3))=1$. Then by Theorem~\ref{bao}, there exist an automorphism $\phi$ of ${\mathbb Z}/23{\mathbb Z}$ and a sign $\epsilon\in\{+1,-1\}$ such that $I_{\phi, \epsilon}(i)$ are even positive numbers for all $i\in {\mathbb Z}/23{\mathbb Z}$. 

No matter which automorphism $\phi$ of ${\mathbb Z}/23{\mathbb Z}$ we choose, we have $I_{\phi, \epsilon}(0)=\epsilon \cdot d(L(23,3), 0)+f(0)=(\epsilon \cdot 3 +11)/2$. In order to make it be an even number, we need $\epsilon$ to be $-1$. By calculation, we see there exist two possible automorphisms of ${\mathbb Z}/23{\mathbb Z}$, the map $\phi_{1}$ given by multiplication by 8 and $\phi_{2}$ given by multiplication by 15, for which $I_{\phi_{j}, -1}(i)=-d(L(23,3), \phi_{j}(i))+f(i)$ are even integers for all $1 \leq i \leq 22$ and $j=1,2$. Precisely in both cases, we have:
\begin{eqnarray*}
\{I_{\phi, -1}(i)\}_{i=0}^{22}=\{4, 0,4,-2,4,0,2,0,2,0,2,0,0,2,0,2,0,2,0,4,-2,4,0\}, 
\end{eqnarray*}
where $\phi$ is either $\phi_{1}$ or $\phi_{2}$. Among these numbers $-2$ is negative. 

As a conclusion, we see that there exists no automorphism $\phi$ of ${\mathbb Z}/23{\mathbb Z}$ and sign $\epsilon\in\{+1,-1\}$ to guarantee that $I_{\phi, \epsilon}(i)$ are all even positive numbers for $i\in {\mathbb Z}/23{\mathbb Z}$. For simplicity, we say there exist no even positive matchings for the knot $S(23,3)$. Therefore the assumption that $u_{2}(S(23,3))=1$ is false, while it is easy to see that $u_{2}(S(23,3))\leq 2$, so finally we have $u_{2}(S(23,3))=2$.

Suppose that $K$ is a $(23,3)$-tangle unknotting number one knot and that $|\det(K)|=23$. Then by Lemma~\ref{lemma1} there is $\Sigma(K)=S^{3}_{23/3}(C)$ for some strongly invertible knot $C$. 

If we get the unknot by changing a negative crossing of $C$ into a positive crossing, then by (\ref{3}) we have:
\begin{eqnarray*}
d(\Sigma(K), i)-d(L(23,3), i) & \geq & 0 \quad \text{and}\\
d(\Sigma(K), i)-d(L(23,3), i)&=& 0 \pmod{2}.
\end{eqnarray*}
Now it is easy to see that there exist no even, negative matchings for the knot $K$ as well. So we conclude that $u_{2}(K)>1$. By Corollary~\ref{cor1}, we have $u_{2}(K)\leq u_{2}(S(3, 23))+1=2$, and therefore $u_{2}(K)=2$. For example, let $C$ be the left-hand trefoil knot. Then in this case $K$ is the knot in Figure~\ref{fig:f9}-(2).

\begin{figure}
	\centering
		\includegraphics[width=0.7\textwidth]{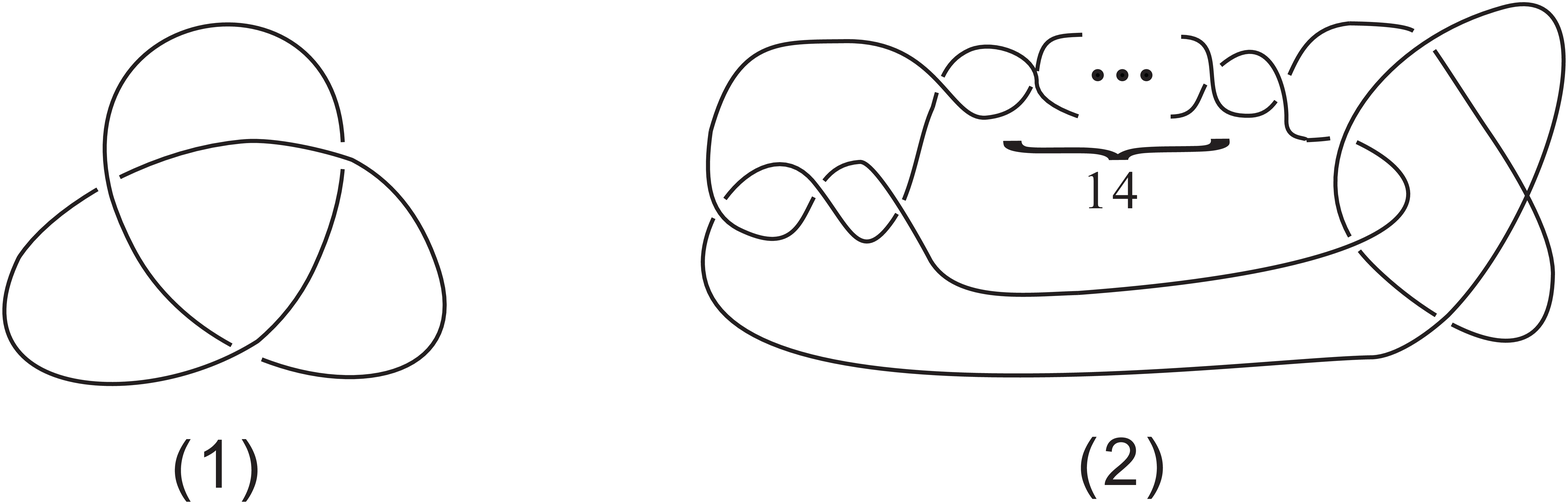}
		\caption{The left-hand trefoil knot $T$ and the tangle unknotting number one knot whose double-branched cover is $S^{3}_{23/3}(T)$.}
	\label{fig:f9}
\end{figure}
\begin{figure}[h]
	\centering
		\includegraphics[width=0.2\textwidth]{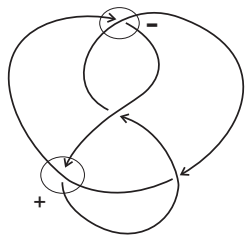}
		\caption{Figure eight knot.}
	\label{fig:f10}
\end{figure}
If there is a positive and a negative crossing in $C$ for which either of the crossing change gives the unknot, and if particularly $C$ is an amphicheiral knot with unknotting number one (For example the figure eight knot in Figure~\ref{fig:f10}), then by (\ref{3}) we have:
\begin{eqnarray*}
d(\Sigma(K), i)-d(L(23,3), i) & = & 0.
\end{eqnarray*}
By a similar argument as in the previous paragraph, we see in this case $u_{2}(K)=2$ as well.

%\begin{rem}
%Since the tangle unknotting numbe one knots have the same linking forms with the corresponding 2-bridge knots, any method related to linking form which does not work for a 2-bridge knot, does not work for the corresponding tangle unknotting number one knot as well.
%\end{rem}

\bibliographystyle{siam}
\bibliography{bao}

\end{document}